\documentclass[10pt,headsepline, a4paper,cleardoubleempty]{scrartcl}

\usepackage[utf8]{inputenc}
\usepackage{geometry}
\usepackage[intlimits]{amsmath}
\usepackage{amsfonts}
\usepackage{amssymb}
\usepackage[amsmath, thmmarks]{ntheorem}
\usepackage{enumerate}
\usepackage[arrow, matrix, curve]{xy}
\usepackage{mathtools}
\usepackage{tikz}
\usepackage{float}
\usepackage{scrpage2}
\usepackage[nottoc]{tocbibind}
\usepackage[pdftoolbar]{hyperref}
\usepackage{pdflscape}
\usepackage{listings}  
\hypersetup{colorlinks=true, linkcolor=blue, citecolor=blue}

\newcommand{\R}{\mathbb{R}}
\newcommand{\Z}{\mathbb{Z}}
\newcommand{\Q}{\mathbb{Q}}
\newcommand{\N}{\mathbb{N}}
\newcommand{\C}{\mathbb{C}}

\newcommand{\transpose}{^{\operatorname{tr}}}
\renewcommand{\sl}{\operatorname{SL}(2,\Z)}

\newcommand{\rank}{\operatorname{rank}}

\makeatother
\newtheorem{vorlage}{}[section]
\newtheorem{prop}[vorlage]{Proposition}
\newtheorem{lem}[vorlage]{Lemma}

\newtheorem{thm}[vorlage]{Theorem}
\theorembodyfont{\rmfamily}
\newtheorem{defin}[vorlage]{Definition}
\newtheorem{rem}[vorlage]{Remark}

\newtheorem{exa}[vorlage]{Example}
\theoremstyle{nonumberbreak}
\theoremsymbol{\ensuremath{\square}}
\newtheorem{proof}{Proof}

\clearscrheadfoot
\rehead{\headmark}
\lehead{\pagemark}
\lohead{\headmark}
\rohead{\pagemark} 

\title{Jacobi forms of theta type}
\author{Martin Woitalla}

\begin{document}
\maketitle
\begin{abstract}
Let $L$ be a positive definite even lattice. We introduce theta type Jacobi forms and construct three towers of Jacobi forms with a particular easy pullback-structure. We use theta type Jacobi forms to explain the existence of a cusp form of weight 24 with respect to the irreducible root lattice $D_{4}$ which is based on additional symmetries of the Coxeter diagram. 
\end{abstract}
\section{Introduction}
\noindent We give a short presentation of the contents. Let $V$ be a real quadratic space of signature $(2,n)$ where $n\in\N$. We will always assume that $n\geq 3$. The bilinear form of $V$ is denoted by $(\cdot,\cdot)$. The group of all isometries of $V$ is called the \textit{orthogonal group of} $V$ and is given by
\[\operatorname{O}(V)=\{g\in\operatorname{GL}(V)\,|\,\forall v\in V \,:\, (gv,gv)=(v,v)\}\,.\]
By $\operatorname{SO}(V)$ we denote the subgroup of index two which equals the kernel of the determinant character. We obtain another subgroup $\operatorname{O}(V)^{+}$ of index two  as the kernel of the real spinor norm. The intersection of the groups $\operatorname{SO}(V)$ and $\operatorname{O}(V)^{+}$ is denoted by $\operatorname{SO}(V)^{+}$. This is the connected component of the identity and is well-known to be a  semisimple and noncompact Lie group, compare for ex. \cite{Kp}. Its maximal compact subgroup is given by $K=\operatorname{SO}(2)\times\operatorname{SO}(n)$.  One can construct a projective model for the Hermitian symmetric space $\operatorname{SO}(V)^{+}/K$ namely 
\[\mathcal{D}=\{[\mathcal{Z}]\in\mathbb{P}(V\otimes\C)\,|\,(\mathcal{Z},\mathcal{Z})=0\,,\,(\mathcal{Z},\overline{\mathcal{Z}})>0\}^{+}\xrightarrow[]{\sim}\operatorname{SO}(V)^{+}/K\,.\] 
Here the superscript $+$ means that we have chosen one of the two connected components.
The associated affine cone is defined as
\[\mathcal{D}^{\bullet}=\{\mathcal{Z}\in V\otimes \C\,|\,[\mathcal{Z}]\in \mathcal{D}\}\,.\]
\begin{defin}\label{def:lattice}
 A subset $L\subseteq V$ is called a \textit{lattice}\index{lattice} in $V$ if there exist linearly independent vectors $b_{1},\dots, b_{N}$ in $V$ such that $L=\Z b_{1}+\dots+\Z b_{N}$. The quantity $N$ is called the rank of $L$ and is denoted by $\rank(L)$. We call $L$ \textit{even} if for all $l\in L$ the condition $(l,l)\in 2\Z$ is satisfied. A lattice is called \textit{a full lattice in} $V$ if $\rank L=\dim V$. A lattice is called positive definite if for all $0\neq l\in L$ we have $(l,l)>0$. For any $t\in\Z$ we define the rescaled lattice $L(t)$ to be the set $L$ equipped with the bilinear form $t(\cdot,\cdot)$.
\end{defin}
  \noindent Let $L_{2}$ be a full lattice in $V$ such that $L_{2}$ is even and splits two hyperbolic planes. In view of the Gramian of $L_{2}$ this means that there exists a basis such that 
\[\operatorname{Gram}(L_{2})=\begin{pmatrix}
					  0&0&0&0&1\\
					  0&0&0&1&0\\   
					  0&0&-S&0&0\\
					  0&1&0&0&0\\
					  1&0&0&0&0
                                         \end{pmatrix}\in\operatorname{GL}(L_{2})\subseteq \operatorname{GL}(V)\,.\]
where $S$ denotes the Gramian of a positive definite even lattice $L$. 
 We consider the arithmetic subgroup 
 \[\operatorname{O}(L_{2})^{+}=\{g\in\operatorname{O}(V)^{+}\,|\,g\,L_{2}\subseteq L_{2}\}\,.\] 
 For any subgroup $\Gamma\leq \operatorname{O}(L_{2})^{+}$ of finite index we consider the modular variety $\Gamma\backslash\mathcal{D}$. This is a noncompact space. We can add cusps to $\Gamma\backslash\mathcal{D}$ such that
\[\Gamma\backslash \mathcal{D}^{\ast}=\Gamma\backslash\mathcal{D}\amalg \coprod_{\Pi}{X_{\Pi}}\amalg \coprod_{l}{Q_{l}}\]
is compact, compare e.g. \cite{BorJ}. Here $l$ runs through the finitely many $\Gamma$-orbits 
of isotropic lines and $\Pi$ runs through the finitely many $\Gamma$-orbits of isotropic planes in $L_{2}\otimes\Q$. Each $X_{\Pi}$ corresponds to a modular curve and is called a one-dimensional cusp of $\Gamma\backslash\mathcal{D}$ and each $Q_{l}$ corresponds to a point and is called a zero-dimensional cusp.  The one-dimensional cusps are represented by rational two-dimensional totally isotropic subspaces of $V$ and the zero-dimensional cusps are represented by primitive isotropic vectors $c\in L_{2}$. 
\begin{defin}\label{def:modular form}
 Let $\Gamma$ be a subgroup of $\operatorname{O}(L_{2})^{+}$. A \textit{modular form} of weight $k\in \Z$ and character $\chi\,:\,\Gamma\to\C^{\times}$ with respect to $\Gamma$ is a holomorphic function  $F\,:\,\mathcal{D}^{\bullet}\to\C$ such that 
\[\begin{aligned}
   &F(t\mathcal{Z})=t^{-k}F(\mathcal{Z})\quad\text{ for all }t\in\C^{\times}\,,&\\
   &F(g\mathcal{Z})=\chi(g)F(\mathcal{Z})\quad\text{ for all }g\in\Gamma\,.&
  \end{aligned}
\]
A modular form is called a \textit{cusp form} if it vanishes at every cusp. The space of modular forms of weight $k$ and character $\chi$ for the group $\Gamma$ will be denoted by $\mathcal{M}_{k}(\Gamma,\chi)$. For the subspace of cusp forms we will write $\mathcal{S}_{k}(\Gamma,\chi)$.  
\end{defin} 
\noindent Orthogonal type modular forms are very important tools to investigate the modular variety  $\Gamma\backslash\mathcal{D}$. In this text we will focus our interest on the Jacobi group $\Gamma^{J}(L)$ which is a distinguished parabolic type subgroup of $\operatorname{O}(L_{2})^{+}$, see section \ref{sec:Jacobi forms} for the definition. Modular forms with respect to $\Gamma^{J}(L)$ are called \textit{Jacobi forms} and can be considered as a natural generalization of the classical Jacobi forms appearing in \cite{EZ}.
\begin{defin}\label{discriminant group and dual lattice}
 Let $L$ be an even lattice in $V$. The dual lattice is the $\Z$-module
 \[L^{\vee}\index{$L^{\vee}$}:=\{x\in V\,|\,\forall l\in L:\,(x,l)\in\Z\}\,.\]
 Since $L\subseteq L^{\vee}$ in this case we can define the discriminant group as the finite abelian group 
 \[D(L):=L^{\vee}/L\,.\]
The lattice $L$ is called maximal if for any even lattice $M$ satisfying both $\operatorname{rank}(L)=\operatorname{rank}(M)$ and $L\subseteq M$ we already have $L=M$.
\end{defin}
\noindent The group $\operatorname{O}(L_{2})^{+}$ acts on the discriminant group $D(L_{2})$. The kernel of this action is denoted by $\widetilde{\operatorname{O}}(L_{2})^{+}$. The following construction is due to Gritsenko, see \cite{G} for the details.
\begin{thm}[Gritsenko'93]
Let $\varphi$ be a Jacobi form with trivial character and weight $k$ where $k\geq 4$.  There exists a linear operator $\operatorname{A-Lift}$ defined on the space of Jacobi forms of weight $k$ into the space of modular forms of weight $k$ such that 
 \[\operatorname{A-Lift}(\varphi)\in \mathcal{M}_{k}(\widetilde{\operatorname{O}}(L_{2})^{+},1)\,.\]
 For maximal even lattices this lift maps cusp forms to cusp forms.
\end{thm}
\noindent In section 2 we will give a review of Jacobi forms in many variables. In section 3 we use Jacobi's theta-series and the weak Jacobi form in two variables $\phi_{0,1}$ of weight $0$ and index $1$ to present pullback constructions for several series of lattices. The forms obtained in this way are called theta type Jacobi forms. In the last section we give a new explanation for the existence of the cusp form of weight 24 for the lattice $D_{4}$ which is based on its representation as an automorphic product defined by theta type Jacobi forms.   

\section{Jacobi forms}\label{sec:Jacobi forms}
\noindent The presentation of the theory is influenced by the classical book on Jacobi forms in one variable of Eichler and Zagier \cite{EZ}. For the description of Jacobi forms in several variables we follow \cite{G} and  \cite{CG}.
Let $V$ be a real quadratic space of signature $(2,2+N)$ where $N\in\N$. Let $L$ be a positive definite even lattice of rank $N$ with bilinear form $(\cdot,\cdot)_{L}=(\cdot,\cdot)$ such that $L(-1)\subseteq V$. By construction $L_{2}\subseteq V$ contains two perpendicular hyperbolic planes $U=\langle e,f \rangle$ and $U_{1}=\langle e_{1},f_{1} \rangle$. Let $F$ be the totally isotropic plane spanned by $e$ and $e_{1}$. We define the parabolic subgroup of $G=\operatorname{SO}(L_{2})^{+}$ fixing $F$ as
\[N_{G}(F)=\{g\in \operatorname{SO}(L_{2})^{+}\,|\,gF=F\}\,.\]
   The matrix realization of $N_{G}(F)$ with respect to our standard basis of $L_{2}$ is well-known.
The \textit{integral Heisenberg group}\index{Heisenberg group} $H(L)$\index{$H(L)$} of the lattice $L$ is generated by the elements 
\[[x,y:r]:=\begin{pmatrix}
            1&0&-y\transpose S&-\frac{1}{2}(x,y)-r&-\frac{1}{2}(y,y)\\
	    0&1&-x\transpose S&-\frac{1}{2}(x,x)&-\frac{1}{2}(x,y)+r\\
	    0&0&I_{N}&x&y\\
	    0&0&0&1&0\\
	    0&0&0&0&1
           \end{pmatrix}\]
where $x,y\in L$ and $r\in\Z/2$ such that $r+\frac{1}{2}(x,y)\in\Z$. The group law of $H(L)$ is 
\[[x,y:r]\cdot
[x',y':r']=[x+x',y+y',r+r'+\frac{1}{2}[(x,y')-(x',y)]]\,.\]
We define an embedding of $\operatorname{SL}(2,\Z)$ into $\operatorname{SO}(L_{2})^{+}$ as
\[\{A\}:=\begin{aligned}
   &\begin{pmatrix}
   A^{\ast}&0&0\\0&I_{N}&0\\0&0&A
  \end{pmatrix}\text{ for }A\in\operatorname{SL}(2,\Z)\,,\,A^{\ast}=\begin{pmatrix}0&1\\1&0\end{pmatrix}(A^{-1})^{\operatorname{tr}}\begin{pmatrix}0&1\\1&0\end{pmatrix}&
  \end{aligned}
\]
The group $\operatorname{SL}(2,\Z)$ acts on the Heisenberg group by conjugation
\[A.[x,y:r]=\{A\}[x,y:r]\{A\}^{-1}=[dx-cy,ay-bx:r]\,.\]
Note that $\frac{1}{2}(dx-cy,ay-bx)=\frac{1}{2}(x,y)\bmod \Z$ since $A\in\sl$.
The \textit{integral Jacobi group}\index{Jacobi group} $\Gamma^{J}(L)$\index{$\Gamma^{J}(L)$} is defined as the subgroup of $N_{G}(F)$ which acts trivially on the sublattice $L$. 
 This group is isomorphic to the semi-direct product 
\[\operatorname{SL}(2,\Z)\ltimes H(L)\,.\] 
 Let $\chi$ be a finite character of the Jacobi group $\Gamma^{J}(L)$. The structure of the Jacobi group implies
\[\chi=v_{\operatorname{SL}(2,\Z)}\cdot\nu_{H(L)}\]
where $v_{\operatorname{SL}(2,\Z)}$ is a finite character of $\operatorname{SL}(2,\Z)$ and $\nu_{H(L)}$ is a finite character for the group $H(L)$. The character $\nu_{H(L)}$ satisfies the equation 
\[\forall A\in\operatorname{SL}(2,\Z)\,\forall [x,y:r]\in H(L)\,:\, \nu_{H(L)}(A.[x,y:r])=\nu_{H(L)}([x,y:r])\,.\]
We recall the definition of \textit{Dedekind`s eta function}\index{Dedekind`s eta function} $\eta\,:\,\mathbb{H}\to\C$\index{$\eta$} where $\mathbb{H}$ denotes the upper half plane in $\C$. This function is defined by the product expansion 
\[\eta(\tau)=q^{1/24}\prod_{n=1}^{\infty}{(1-q^{n})}\quad,\quad q=e^{2\pi i\tau}\,.\] 
The transformation law for the generators of $\operatorname{SL}(2,\Z)$ is 
\[\begin{aligned}
   &\eta(\tau+1)=e^{\pi i/12}\eta(\tau)&\\
   &\eta(-1/\tau)=(-i\tau)^{1/2}\eta(\tau)&
  \end{aligned}\]
where we have chosen the branch of the square root $z^{1/2}$ which is positive if $z>0$.
There exists a multiplier system $v_{\eta}\,:\,\mathbb{H}\to\C^{\ast}$ of order 24 whose square is a finite character satisfying 
\[\eta\left(\frac{a\tau +b}{c\tau+d}\right)=v_{\eta}(A)\,\eta(\tau)\quad\text{for}\quad A=\begin{pmatrix}
a&b\\c&d                                                                                                                                                                                                                                                                                                                                                                                                            \end{pmatrix}
\in\operatorname{SL}(2,\Z)\,.\]
For more details on the function $\eta(\tau)$ we refer to \cite{Ap}, section 3.    
We define the quantities $s(L),n(L)\in\N$ as the generators of the ideal in $\Z$ generated by $(x,y)$ or $(x,x)$, respectively, where $x,y\in L$. These quantities are called \textit{scale}\index{lattice!scale of} and \textit{norm}\index{lattice!norm of} of the lattice $L$. Obviously $s(L)|n(L)\,.$ For a proof of the next Proposition see  \cite{CG} and \cite{R}.
\begin{prop}\label{prop:characters of Jaobi group}
 \begin{enumerate}[(a)]
  \item  The group of finite characters of $\operatorname{SL}(2,\Z)$ is a cyclic group of order 12 and is generated by $v_{\eta}^{2}$.
 \item Let $\nu_{H(L)}\,:\,H(L)\to\C^{\ast}$ be a character of finite order such that 
\[\nu_{H(L)}(A.[x,y:r])=\nu_{H(L)}([x,y:r])\]
 for all $[x,y:r]\in H(L)$ and $A\in\operatorname{SL}(2,\Z)\,.$ Then $\nu_{H(L)}$ is of the shape
\[\nu_{H(L)}([x,y:r])=e^{\pi i t((x,x)+(y,y)-(x,y)+2r)}\]
where $t\in\Q$ such that $t\cdot s(L)\in\Z$. 
 \end{enumerate}
\end{prop}
\noindent  We extend $(\cdot,\cdot)$  to $L \otimes \C$ by $\C$-linearity. Let $k\in\Z$ and $t\in\Q$. We define an action of the Jacobi group on the space of holomorphic functions on $\mathbb{H}\times (L\otimes \C)$ by means of the generators of $\Gamma^{J}(L)$:
\begin{equation}\label{definition action of Jacobi group}
 \begin{aligned}
 &(\left.\varphi\right|_{k,t}A)(\tau,\mathfrak{z})&:=&&&(c\tau +d)^{-k}e^{-\pi i t\frac{c(\mathfrak{z},\mathfrak{z})}{c\tau+d}}\cdot\varphi\left(\frac{a\tau+b}{c\tau+d},\frac{\mathfrak{z}}{c\tau +d}\right) &\\
&(\left.\varphi\right|_{k,t}[x,y:r])(\tau,\mathfrak{z})&:=&&&e^{2\pi it(\frac{1}{2}(x,x)\tau+(x,\mathfrak{z})+\frac{1}{2}(x,y)+r)}\cdot \varphi(\tau,\mathfrak{z}+x\tau+y)&
\end{aligned}
\end{equation}
where $A=\begin{pmatrix}                                                                                                                                                                                                       a&b\\c&d                                                                                                                                                                                                             \end{pmatrix}\in\sl$, $[x,y:r]\in H(L)$ and $(\tau,\mathfrak{z})\in \mathbb{H}\times (L\otimes \C)$.  The two assignments jointly define an action of the Jacobi group on the space of holomorphic functions. Let $\varphi$ be a holomorphic function satisfying 
\[\forall g\in\Gamma^{J}(L)\,:\,\left.\varphi\right|_{k,t}g=\chi(g)\,\varphi\]
where $\chi$ is a finite character for the Jacobi group. 
Then $\varphi$ is periodic in $\tau$ and $\mathfrak{z}$ since the identities
\[\begin{aligned}
   &\varphi(\tau+1,\mathfrak{z})=e^{2\pi i\frac{D}{24}}\varphi(\tau,\mathfrak{z})\,,\,\varphi(\tau,\mathfrak{z}+2y)=\nu(0,2y:0)\, \varphi(\tau,\mathfrak{z})=\varphi(\tau,\mathfrak{z})&
  \end{aligned}\]
hold true for any $y\in L$ and for some $D\in\N$. Thus the function has a Fourier expansion of the shape
\begin{equation}\label{Fourier expansion of weak Jacobi form}
 \varphi(\tau,\mathfrak{z})=\sum_{\substack{n\in\Q,n\equiv \frac{D}{24}\bmod \Z\\l\in\frac{1}{2}L^{\vee}}}{f(n,l)e^{2\pi i (n\tau+(l,\mathfrak{z}))}}\,.
\end{equation}
We now introduce the notion of a Jacobi form.
\begin{defin}\label{def: Jacobi forms}
 Let $k\in\Z$ and $t\in\Q_{\geq 0}$. A holomorphic function 
\[\varphi\,:\,\mathbb{H}\times(L\otimes \C)\to\C\]
 is called a \textit{weak Jacobi form}\index{Jacobi form} of weight $k$ and index $t$ with character $\chi$ if it satisfies
\[\forall g\in\Gamma^{J}(L)\,:\,\left.\varphi\right|_{k,t}g=\chi(g)\,\varphi\]
and has a Fourier expansion as in (\ref{Fourier expansion of weak Jacobi form}) where additionally
\[f(n,l)\neq 0\implies n\geq 0\,.\]
 We call $\varphi$ a \textit{holomorphic Jacobi form} if the Fourier expansion ranges over all $n,l$ with $2nt-(l,l)\geq 0$ and $\varphi$ is called  a \textit{Jacobi cusp form} if it ranges over all $n,l$ satisfying $2nt-(l,l)> 0$. We denote by $J^{\textit{(weak)}}_{k,L;t}(\chi)$\index{$J^{\textit{(weak)}}_{k,L;t}(\chi)$} the vector space of weak Jacobi forms and the corresponding spaces of holomorphic forms and cusp forms are denoted by $J_{k,L;t}(\chi)$\index{$J_{k,L;t}(\chi)$} and $J^{\textit{(cusp)}}_{k,L;t}(\chi)$\index{$J^{\textit{(cusp)}}_{k,L;t}(\chi)$}, respectively. If the character is trivial we also write $J_{k,L;t}=J_{k,L;t}(1)$\index{$J_{k,L;t}$} for short.
\end{defin}
\begin{rem}
\begin{enumerate}[(a)]
 \item The action can be extended for $k\in\frac{1}{2}\Z$ and $\chi|_{\sl}$ being a multiplier system for $\sl$. Here we have to replace $\sl$ by the metaplectic cover $\operatorname{Mp}(2,\Z)$, see e.g. \cite{Br}. In this situation we also use the notation $J^{(\ast)}_{k,L;t}(\chi)$ for the spaces of Jacobi forms. 
\item A Jacobi form is called \textit{singular} if it has weight $N/2$. 
 \item The notion of a Jacobi form is compatible with Definition \ref{def:modular form}. To see this we note that we have an affine model for the homogeneous domain $\mathcal{D}$ given by
\begin{equation}\label{affine model for hom domain}
\mathcal{H}(L_{2})=
\left\{\begin{pmatrix}
         \omega\\\mathfrak{z}\\\tau
        \end{pmatrix}
\in \C\times (L\otimes \C)\times \C\,\left|\begin{aligned}
                                                                                                  &\omega_{i},\tau_{i}>0\,,&\\&2\omega_{i}\tau_{i}-(\mathfrak{z}_{i},\mathfrak{z}_{i})>0
                                                                                                 \end{aligned}\right.\right\}
\end{equation}
where we have used the abbreviations
\[\omega_{i}:=\operatorname{Im}(\omega)\quad,\quad \tau_{i}:=\operatorname{Im}(\tau)\quad,\quad \mathfrak{z}_{i}:=\operatorname{Im}(\mathfrak{z})\,.\]
 Let $\varphi\in J_{k,L;t}(\chi)$ where we assume $k\in\Z$. We define a holomorphic function  on $\mathcal{H}(L_{2})$ by
 \[\widetilde{\varphi}(\tau,\mathfrak{z})=\varphi(\tau,\mathfrak{z})e^{2\pi it\omega}\,.\]
Since  $\mathcal{D}$ and  $\mathcal{H}(L_{2})$ are biholomorphically equivalent we can interpret $\widetilde{\varphi}$ as an element of $\mathcal{M}_{k}(\Gamma^{J}(L),\chi)$.  
 \item Suppose $\varphi$ is a non-vanishing Jacobi form of index $t$ and weight $k$ with character $\chi=v_{\eta}^{D}\cdot\nu$. Consider the extended Jacobi form $\widetilde{\varphi}(\tau,\mathfrak{z})=\varphi(\tau,\mathfrak{z})e^{2\pi it\omega}$. By part (b) we know that $\widetilde{\varphi}$ can be interpreted as a modular form. 
If we transform this function by  $g=[0,0:1]\in H(L)$ an investigation of the modular behaviour on the affine domain $\mathcal{H}(L_{2})$ yields $e^{2\pi i t}=\nu([0,0:1])$ and Proposition \ref{prop:characters of Jaobi group} implies  
\[\nu([x,y:r])=e^{\pi i t((x,x)+(y,y)-(x,y)+2r)}\,.\] 
 \item  The Jacobi forms for a fixed lattice $L$ and with trivial character form a bigraded ring. The transformation behaviour of elements belonging to $J_{k_{1},L;t_{1}}J_{k_{1},L;t_{2}}$ is that of a Jacobi form of weight $k_{1}+k_{2}$ and index $t_{1}+t_{2}$. If both factors are holomorphic then the product is also a holomorphic Jacobi form by virtue of the identity
\[\begin{aligned}
   &k_{1}+k_{2}-\frac{(x+y,x+y)}{2(t_{1}+t_{2})}&\\
  &=\left(k_{1}-\frac{(x,x)}{2t_{1}}\right)+\left(k_{2}-\frac{(y,y)}{2t_{2}}\right)+\frac{(t_{1}y-t_{2}x,t_{1}y-t_{2}x)}{2t_{1}t_{2}(t_{1}+t_{2})}&
  \end{aligned}\]
which holds for all $k_{1},k_{2}\in\Q_{\geq 0}$\,,\,$t_{1},t_{2}>0$ and $x,y\in L\otimes \R$\,. 
\end{enumerate}
\end{rem}
\noindent We finish this section with several examples of Jacobi forms for the root lattice $A_{1}\cong(\Z,2x^{2})$ which correspond to Jacobi forms in the sense of \cite{EZ}. In the sequel we denote by $(\cdot,\cdot)_{N}$ the standard scalar product in $\R^{N}$. The standard basis of $\R^{N}$ is denoted by
\[\varepsilon_{1},\dots,\varepsilon_{N}\,.\] 
For $\tau\in\mathbb{H}$ we set $q:=\exp(2\pi i\tau)$.
\begin{exa}\label{eta function and Gritsenkos theta function}
 \begin{enumerate}[(a)]
  \item Dedekind's eta function is a Jacobi form of weight $1/2$ and index $0$ for every positive definite even lattice $L$, thus $\eta\in J_{1/2,L;0}^{\textit{(cusp)}}(v_{\eta})$ where $v_{\eta}$ is not a character but a multiplier system for $\sl$. For $m\in\Z$ one defines
\[\left(\frac{-4}{m}\right):=\begin{cases}
                        1&,\, m\equiv 1\bmod 4\\
                        -1&,\, m\equiv -1\bmod 4\\
                        0&,\, m\equiv 0\bmod 2
                       \end{cases}\,.\]
 We want to mention the following infinite sum-expansion for $\eta^{3}(\tau)$ 
 \begin{equation}\label{representation of eta to the cube}
  \eta^{3}(\tau)=\frac{1}{2}\sum_{n\in\Z\,,\, n\equiv 1 \bmod 2}{\left( \frac{-4}{n}\right)nq^{\frac{n^{2}}{8}}}
 \end{equation}
 which follows from the Jacobi triple identity.
 \item In \cite{EZ} the authors constructed the Jacobi-Eisenstein series $e_{k,1}\in J_{k,A_{1};1}$\index{$e_{k,1}$}. One defines a Jacobi cusp form $\varphi_{12,A_{1}}\in J_{12,A_{1};1}$ as \index{$\varphi_{12,A_{1}}$}
\[\varphi_{12,A_{1}}(\tau,z):=\frac{1}{144}(e_{4}^{2}(\tau)e_{4,1}(\tau,z)-e_{6}(\tau)e_{6,1}(\tau,z))\]
where
\[e_{k}(\tau)=1-\frac{2k}{B_{k}}\cdot\sum_{n=1}^{\infty}{\sigma_{k-1}(n)q^{n}}\]
\index{$e_{k}$} denotes the classical Eisenstein series of weight $k$ for the group $\sl$, confer e.g. \cite{KK}, p. 161.
 \item  We define
\[\phi_{0,1}(\tau,z):=\frac{\varphi_{12,A_{1}}(\tau,z)}{\eta^{24}(\tau)}=\frac{e_{4}^{2}(\tau)\,e_{4,1}(\tau,z)-e_{6}(\tau)\,e_{6,1}(\tau,z)}{\frac{1}{12} \,(e_{4}^{3}(\tau)-e_{6}^{2}(\tau))}\,.\]
\index{$\varphi_{0,1}$}
 In accordance with the classical theory of elliptic modular forms we write
\[\Delta_{12}(\tau):=\eta^{24}(\tau)\]
\index{$\Delta_{12}$} for the first cusp form for $\sl$.
One has $\phi_{0,1}\in J^{\textit{(weak)}}_{0,A_{1};1}(1)$ since
\[\phi_{0,1}(\tau,z)=\sum_{\substack{n,s\in\Z, n\geq 0\\4n-s^{2}\geq -1}}{q^{n}r^{s}}=(r+10+r^{-1})+q(\dots)\]
regarding the convention
\[r=\exp(2\pi i z)\,.\]
Moreover the function defined by $(\tau,z)\mapsto\eta^{6}(\tau)\phi_{0,1}(\tau,z)$ is a holomorphic Jacobi form of weight $3$ and index 1 for the lattice $A_{1}$ with Fourier expansion 
\[\eta^{6}(\tau)\phi_{0,1}(\tau,z)=q^{\frac{1}{4}}(r+10+r^{-1})+q^{\frac{5}{4}}(\cdots)\,.\]
  \item The Jacobi theta-series of characteristic $(\frac{1}{2},\frac{1}{2})$ is given by 
\[\vartheta(\tau,z)\index{$\vartheta(\tau,z)$}=\sum_{n\in\Z}{\left( \frac{-4}{n}\right)q^{\frac{n^{2}}{8}}r^{\frac{n}{2}}}=\sum_{n\in\Z\,,\,n \equiv 1\bmod 2}{(-1)^{\frac{n-1}{2}}\exp\left(\frac{\pi i n^{2}\tau}{4}+\pi i n z\right)}\]
where $q=e^{2\pi i \tau}\,,\,\tau\in\mathbb{H}$ and $r=e^{2\pi i z}\,,\,z\in\C$. 
This function was originally discovered by Carl Gustav Jacob Jacobi. In \cite{GN} the authors reinterpreted this function as a modular form of half-integral weight and index. The same proof as in \cite{KK}, section I.6 can be used to show that $\vartheta$ defines a holomorphic function on $\mathbb{H}\times \C$. 
The transformation behaviour under the action of $H(A_{1})$ is
\[\vartheta(\tau,z+x\tau+y)=(-1)^{x+y}\,\exp(-\pi i(x^{2}\tau+2 xz))\vartheta(\tau,z)\,.\]
Moreover we have
\[\vartheta(\tau+1,z)=\exp(\pi i/4)\vartheta(\tau,z)\]
and
\[\vartheta(-1/\tau,z/\tau)=\exp(-3\pi i/4)\tau^{1/2}\exp(\pi i z^{2}/\tau)\vartheta(\tau,z)\]
by the theta transformation formula. This shows that $\vartheta$ transforms like a Jacobi form of weight $1/2$ and index $1/2$. If we consider the extended Jacobi form $\widetilde{\vartheta}(\tau,z,\omega)=\vartheta(\tau,z)e^{\pi i \omega}$ we obtain
\[\begin{aligned}
   &\left.\widetilde{\vartheta}(\tau, z,\omega)\right|_{\frac{1}{2}}\,[x,y:r]&&=e^{\pi i(x+y)}\,\exp(-\pi i(x^{2}\tau+2 xz))\vartheta(\tau,z)&\\
&&&\times \exp(\pi i (\omega+2xz+x^{2}\tau+xy+r))&\\
  &&&=e^{\pi i(x+y+xy+r)}\,\widetilde{\vartheta}(\tau,z,\omega)&\\
&&&=:\nu_{H(A_{1})}([x,y:r])\,\widetilde{\vartheta}(\tau,z,\omega)&
  \end{aligned}\] 
From the infinite sum expansion we obtain
\[\left.\frac{\partial\vartheta(\tau,z)}{\partial z}\right|_{z=0}=\pi i\sum_{n\in\Z\,,\, n\equiv 1 \bmod 2}{\left( \frac{-4}{n}\right)nq^{\frac{n^{2}}{8}}}=2\pi i\eta(\tau)^{3}\] 
by (\ref{representation of eta to the cube}). This shows that $\vartheta$ transforms like a Jacobi form with mutiplier system $v_{\eta}^{3}\times \nu_{H(A_{1})}$\,. 
Moreover Jacobi's triple identity yields
\[\vartheta(\tau,z)=-q^{1/8}r^{-1/2}\prod_{n\geq 1}{(1-q^{n-1}r)(1-q^{n}r^{-1})(1-q^{n})}\,.\]
Since each summand in the infinite sum defining $\vartheta$ is of the shape $e^{2\pi i \tau n^{2}/8}e^{2\pi i\cdot 2(z,m/4)_{1}}$ we see that $\vartheta$ is a holomorphic Jacobi form. 
The function has the properties
\[\begin{aligned}
   &\vartheta(\tau,-z)&&=&&-\vartheta(\tau,z)\,,&\\&\vartheta(\tau,z+x\tau+y)&&=&&(-1)^{x+y}\exp(-\pi i(x^{2}\tau+2xz))\,\vartheta(\tau,z)&
  \end{aligned}\]
for all $x,y\in\Z$\,.
Hence the divisor of $\vartheta$ contains $\{z=x\tau+y\,|\,x,y\in\Z\}\,.$ Finally a classical argument from the theory of complex functions yields 
\[\operatorname{div}(\vartheta)=\{x\tau+y\,|\,x,y\in\Z\}\] compare for example \cite{FB}, section 5.6.  
 \end{enumerate}
\end{exa}

\section{Theta type Jacobi forms for several root lattices}\label{sec: theta type Jacobi forms}
\noindent In this section we give several pullback constructions for Jacobi forms. The constructions are motivated by \cite{G1} where the author constructed towers of reflective modular forms by means of Jacobi forms. All lattices appearing in this section are even. The next Lemma can be found in \cite{CG}, Proposition 3.1.
\begin{lem}
 Let $L\leq M$ be a sublattice of $M$ such that $\rank L < \rank M$ and $\varphi\in J_{k,M,t}(\chi)$ be a Jacobi form of weight $k$ and index $t$ for the character $\chi$. Consider the decomposition $\mathfrak{z}_{M}=\mathfrak{z}_{L}\oplus\mathfrak{z}_{L^{\perp}}\in M\otimes\C=(L\oplus (L)_{M}^{\perp})\otimes \C$. We define the pullback of $\varphi$ to $L$ as the function  $\varphi\downharpoonright_{L}$ on $\mathbb{H}\times
 (L\otimes\C)$ defined by
\[\varphi\downharpoonright_{L}(\tau,\mathfrak{z}_{L}):=\varphi(\tau,\mathfrak{z}_{L}\oplus 0)\,.\]
Then $\varphi\downharpoonright_{L}\in  J_{k,L,t}(\left.\chi\right|_{L})$ and the pullback maps cusp forms to cusp forms.
\end{lem}
\begin{defin}\label{def:pullbak notation}
 Let $L\leq M$ be a sublattice of $M$ such that $\rank L < \rank M$ and $\varphi\in J_{k,M,t}^{\textit{(weak)}}$ be a Jacobi form of weight $k$ and index $t$.  Let $\psi\in J_{k,L;t}^{\textit{(weak)}}$. We say that $\psi$ is a pullback of $\varphi$ if there exists some $\alpha\in \C^{\times}$ such that 
\(\psi=\alpha \cdot\varphi\downharpoonright_{L}\,.\)
In this case we use the notation $\varphi\to \psi$\,. We set $\varphi\downharpoonright_{L}:=\varphi$ if $\rank L =\rank M$.
\end{defin}
\noindent Let $L\subseteq \R^{m}$ be a positive definite even lattice. We define
\[\vartheta_{L}\,:\,\mathbb{H}\times(L\otimes\C)\to\C\quad,\quad \vartheta_{L}(\tau,\mathfrak{z}):=\prod_{j=1}^{m}{\vartheta(\tau,(\mathfrak{z},\varepsilon_{j}))}\,.\]
\noindent We now introduce Jacobi forms of theta type.
\begin{defin}\label{def:theta type JF}
 Let $L\subseteq \R^{m}$ be a positive definite even lattice and $\varphi\in J_{k,L;t}$. We say that $\varphi$ is of theta type if there exists a sublattice $L'\subseteq L$, $\alpha\in\C^{\times}$ and integers $a,b\in\Z_{\geq 0}$ such that
\[\varphi\downharpoonright_{L'} (\tau,\mathfrak{z}')=\alpha\cdot\eta(\tau)^{a}\,\vartheta_{L'}(\tau,\mathfrak{z}')^{b}\,.\] 
\end{defin}
\begin{prop}\label{prop:theta type construction for arbitrary root lattice}
 Let $L\subseteq \R^{m}$ be a positive definite even lattice. We assume that its bilinear form is $(\cdot,\cdot):=(\cdot,\cdot)_{m}$.  
\begin{enumerate}[(a)]
 \item The function $\vartheta_{L}$ transforms like a Jacobi form of index $1$ and weight $m/2$ for the lattice $L$ with a character having trivial Heisenberg part. The Fourier expansion of this function is of the shape
\[ \vartheta_{L}(\tau,\mathfrak{z})=\sum_{\substack{n\in\Q_{\geq 0},n\equiv \frac{D}{24}\bmod \Z\\l\in\frac{1}{2}\Z^{m}\\2n-(l,l)= 0}}{f(n,l)e^{2\pi i (n\tau+(l,\mathfrak{z}))}}\]
for some $D\in\N$.
 \item Let $L\leq M$ and suppose $J\subseteq\{1,\dots,m\}$ such that $M\subseteq \langle \varepsilon_{j}|j\in J\rangle_{\R}$. We assume that  there exists a $\kappa\in J$ such that 
\[L=M\cap \langle \varepsilon_{j}|j\in J\,,\,j\neq \kappa\rangle_{\R}\,.\]
We define a function $\prescript{}{L(2)}{\prod}^{M(2)}(\vartheta_{L}^{2},\phi_{0,1})\,:\,\mathbb{H}\times (M(2)\otimes \C)\to\C$ by
\[\prescript{}{L(2)}{\prod}^{M(2)}(\vartheta_{L}^{2},\phi_{0,1})(\tau,\mathfrak{z}):=\phi_{0,1}(\tau,(\mathfrak{z},\varepsilon_{\kappa}))\cdot\vartheta_{L}^{2}(\tau,\mathfrak{z})\,.\]
Then $\prescript{}{L(2)}{\prod}^{M(2)}(\vartheta_{L}^{2},\phi_{0,1})$ transforms like a Jacobi form of index 2 for the lattice $M$ with a character having trivial Heisenberg part and satisfies $\prescript{}{L(2)}{\prod}^{M(2)}(\vartheta_{L}^{2},\phi_{0,1})\to\vartheta_{L}^{2}$.
\end{enumerate}
\end{prop}
\begin{proof}
\begin{enumerate}[(a)]
 \item   For any $\mathfrak{z}\in L\otimes \C$ the identity 
\[\mathfrak{z}=\sum_{j=1}^{m}{(\mathfrak{z},\varepsilon_{j})\varepsilon_{j}}\]
is true.  
Hence $(\mathfrak{z},\mathfrak{z})=\sum_{j=1}^{m}{(\mathfrak{z},\varepsilon_{j})^{2}}$ and for any $x\in L$ we obtain
\[\sum_{j=1}^{m}{(x,\varepsilon_{j})}=\sum_{j=1}^{m}{(x,\varepsilon_{j})^{2}}=(x,x)=0\bmod 2\Z\,.\]
From Example \ref{eta function and Gritsenkos theta function} (d) we deduce that \(\vartheta_{L}\) has the transformation behaviour and character claimed above. The same example yields that each summand of the Fourier expansion is of the shape
\[\alpha\prod_{j\in J}{\left(q^{\frac{n_{j}^{2}}{8}}\exp(\pi i n_{j}(\mathfrak{z},\varepsilon_{j}))\right)}\]
where $J\subseteq \{1,\dots, m\}\,,\,n_{j}\in\Z$ and $\alpha \in \Z$. Note that $\sharp J\geq \operatorname{rank}(L)$. Now define 
\[v:=\frac{1}{2}\sum_{j\in J}{n_{j}\varepsilon_{j}}\in \frac{1}{2}\Z^{m}\]
to write the above summand as
\[\alpha\exp(2\pi i (\mathfrak{z},v))\cdot\prod_{j\in J}{q^{\frac{n_{j}^{2}}{8}}}\,.\]
This completes the proof of part (a).

\item The statement follows from the identity
\[\prescript{}{L(2)}{\prod}^{M(2)}(\vartheta_{L}^{2},\phi_{0,1})(\tau,\mathfrak{z})=\phi_{0,1}(\tau,(\mathfrak{z},\varepsilon_{\kappa}))\cdot\prod_{j\in J\backslash \{\kappa\}}{\vartheta^{2}(\tau,(\mathfrak{z},\varepsilon_{j}))}\]
and an analogue investigation as in part (a). The statement on the pullback follows from  the facts $\phi_{0,1}(\tau,0)=12$ and $(\mathfrak{z}_{L},\varepsilon_{\kappa})=0$ if $\mathfrak{z}_{L}\in L\otimes\C$.
\end{enumerate}
\end{proof}
\noindent We now restrict ourselves to some special types of (root) lattices. We first note that for any $t,k\in\N$ the space of classical Jacobi forms of weight $k$ and index $t$ as defined in \cite{EZ} coincides with the space $J_{k,A_{1}(t);1}=J_{k,A_{1};t^{2}}$. For the next Lemma we consider the lattice $NA_{1}$ which can be realized as $\Z^{N}$ equipped with the bilinear form $2(\cdot,\cdot)_{N}$. Note that every $\mathfrak{z}\in NA_{1}\otimes\C$ can be written in the form 
\begin{equation}\label{choice of coordinates for NA1}
\mathfrak{z}=\sum_{j=1}^{N}{z_{j}\varepsilon_{j}}\quad,\quad z_{j}\in\C\,. 
\end{equation}
\begin{lem}\label{lemma:tensor construction for NA1}
 Let $N,t,k\in\N$ and $\varphi\in J_{k,NA_{1},t^{2}}^{\textit{(weak)}}$. We define a function
\[\varphi\otimes\phi_{0,1}\,:\,\mathbb{H}\times ((N+1)A_{1}\otimes\C)\to\C \]
where
\[\varphi\otimes\phi_{0,1}(\tau,\mathfrak{z}_{(N+1)A_{1}}):=\varphi(\tau,\mathfrak{z}_{NA_{1}})\phi_{0,1}(\tau,tz_{N+1})\]
is an element of $J^{(\textit{weak})}_{k,(N+1)A_{1};t^{2}}$. Moreover we have
\[\varphi\otimes\phi_{0,1}\to\varphi\,.\]
\end{lem}
\begin{proof}
 Let $\tau\in \mathbb{H},\mathfrak{z}\in ((N+1)A_{1}\otimes\C)$ and $x,y\in (N+1)A_{1}$. We consider the decomposition for $\mathfrak{z}$
\[\mathfrak{z}=\mathfrak{z}_{NA_{1}}\oplus z_{N+1}\varepsilon_{N+1}\]
and decompose $x,y$ in the same manner. 
For every $A\in\sl$ and $x,y\in \Z$ one has
\[\begin{aligned}
   &\phi_{0,1}\left(\frac{a\tau+b}{c\tau +d},\frac{tz_{N+1}}{c\tau+d}\right)=&\\&\qquad\qquad\qquad(c\tau+d)^{k}\exp\left(2t^{2}\pi i\frac{z_{n+1}^{2}}{c\tau+d}\right)\phi_{0,1}(\tau,tz_{N+1})&\\
   &\phi_{0,1}(\tau,tz_{N+1}+tx_{N+1}\tau+ty_{N+1})=&\\&\qquad\qquad\qquad\exp(-2\pi it^{2}(x_{N+1}^{2}+2x_{N+1}z_{N+1}))\,\phi_{0,1}(\tau,tz_{N+1})&
  \end{aligned}\]
which shows that $\varphi\otimes\phi_{0,1}$ has the correct transformation behaviour. Since $NA_{1}(t^{2})\subseteq \R^{N}$ we conclude that $\varphi\otimes\phi_{0,1}$ has a Fourier expansion of the shape 
\[ \varphi\otimes\phi_{0,1}(\tau,\mathfrak{z})=\sum_{\substack{n\in\N\\l\in\frac{1}{2t^{2}}\Z^{N}}}{f(n,l)e^{2\pi i (n\tau+(l,\mathfrak{z}))}}\]
where $(\cdot,\cdot):=2t^{2}(\cdot,\cdot)_{N}$.
\end{proof}
\noindent Following \cite{G1} and \cite{CG} we can construct Jacobi forms for the special case $N\leq 4$ and $t=1$. \newpage
\begin{prop}\label{prop:singular theta series for 4A1}
Let $N\in\{1,\dots, 4\}$ and consider the coordinates (\ref{choice of coordinates for NA1}).
 \begin{enumerate}[(a)]
\item We have $\vartheta_{NA_{1}}\in J_{N/2,NA_{1};\frac{1}{2}}(v_{\eta}^{3N}\times \nu_{H(NA_{1})})$
where 
\[\nu_{H(NA_{1})}([x,y:r])=(-1)^{r+\sum_{j=1}^{N}{x_{j}+y_{j}+x_{j}y_{j}}}\quad \text{for all }[x,y:r]\in H(NA_{1})\,.\]
 \item We define the functions
\[\begin{aligned}
&\varrho_{6-N,NA_{1}},\psi_{12-2N,NA_{1}}\,:\,\mathbb{H}\times (NA_{1}\otimes \C)\to\C\,,&\\
& \varrho_{6-N,NA_{1}}(\tau,\mathfrak{z}_{NA_{1}}):=\eta(\tau)^{12-3N}\vartheta_{NA_{1}}(\tau,\mathfrak{z}_{NA_{1}})\,,&\\
&\psi_{12-2N,NA_{1}}:=\varrho_{6-N,NA_{1}}^{2}\,.&   
  \end{aligned}\]
Then  $\varrho_{6-N,NA_{1}}$ belongs to $J_{6-N,NA_{1};\frac{1}{2}}(v_{\eta}^{12}\times \nu_{H(NA_{1})})$ and $\psi_{12-2N,NA_{1}}$ belongs to $J_{12-2N,NA_{1};1}$. 
 \end{enumerate}
\end{prop}
\begin{proof}
Since by our construction the ambient vector space of $NA_{1}$ is $\R^{N}$ we can derive the holomorphicity of  $\vartheta_{NA_{1}}$ by the Fourier expansion given in Proposition \ref{prop:theta type construction for arbitrary root lattice} part (a).  The Heisenberg part of these functions is a binary character since the bilinear form of $NA_{1}$ is $2(\cdot,\cdot)_{N}$ instead of $(\cdot,\cdot)_{N}$. The last fact is also responsible for the appearance of a half-integral index in this case. The other statements are direct consquences of  Proposition \ref{prop:theta type construction for arbitrary root lattice} Example \ref{eta function and Gritsenkos theta function} part (a) and (d). 
\end{proof}
\noindent The previous considerations show that there exists an operator
\begin{equation}
 \otimes\phi_{0,1}\index{$\otimes_{NA_{1}}\phi_{0,1}$}\,:\,J^{\textit{(weak)}}_{k,NA_{1};1}\to J^{\textit{(weak)}}_{k,(N+1)A_{1};1}\quad,\quad \varphi\mapsto\varphi\otimes\phi_{0,1}
\end{equation}
which extends to Jacobi forms of index zero being isomorphic to the space of elliptic modular forms  $\mathcal{M}_{k}(\sl)$ of weight $k$ by the same construction. 
This enables us to build a tower of theta type Jacobi forms for $4A_{1}$.
\begin{prop}\label{prop:tower for 4A1}
Let $N=1,\dots,4$. We have the following diagram of Jacobi forms
 \begin{center}
 \begin{tikzpicture}[scale=0.5,node distance=2cm,auto]
  \node[] (delta) at (0cm,0cm) {$\textcolor{black}{\Delta_{12}}$};
   \node (12+A1) at (0cm,-2.5cm) {$\textcolor{black}{\varphi_{12,A_{1}}}$};
   \node (12+2A1) at (0cm,-5cm) {$\textcolor{black}{\varphi_{12,2A_{1}}}$};
   \node (12+3A1) at (0cm,-7.5cm) {$\textcolor{black}{\varphi_{12,3A_{1}}}$};
   \node (12+4A1) at (0cm,-10cm) {$\varphi_{12,4A_{1}}$};
   \node[] (10+A1) at (5cm,-2.5cm) {$\textcolor{black}{\psi_{10,A_{1}}}$};
   \node[] (10+2A1) at (5cm,-5cm) {$\textcolor{black}{\varphi_{10,2A_{1}}}$};
   \node[] (10+3A1) at (5cm,-7.5cm) {$\textcolor{black}{\varphi_{10,3A_{1}}}$};
   \node[] (10+4A1) at (5cm,-10cm) {$\varphi_{10,4A_{1}}$};
   \node[] (8+2A1) at (10cm,-5cm) {$\textcolor{black}{\psi_{8,2A_{1}}}$};
   \node[] (8+3A1) at (10cm,-7.5cm) {$\textcolor{black}{\varphi_{8,3A_{1}}}$};
   \node[] (8+4A1) at (10cm,-10cm) {$\varphi_{8,4A_{1}}$};
   \node[] (6+3A1) at (15cm,-7.5cm) {$\textcolor{black}{\psi_{6,3A_{1}}}$};
  \node[] (6+4A1) at (15cm,-10cm) {$\varphi_{6,4A_{1}}$};
  \node[] (4+4A1) at (20cm,-10cm) {$\psi_{4,4A_{1}}$};
  \node[] () at (19.0cm,-8.00cm) {$\left.\frac{\partial^{2}}{\partial z_{4}^{2}}\right|_{z_{4}=0}$};
  \node[] () at (14.0cm,-5.50cm) {$\left.\frac{\partial^{2}}{\partial z_{3}^{2}}\right|_{z_{3}=0}$};
  \node[] () at (9.0cm,-3.00cm) {$\left.\frac{\partial^{2}}{\partial z_{2}^{2}}\right|_{z_{2}=0}$};
  \node[] () at (4.0cm,-0.50cm) {$\left.\frac{\partial^{2}}{\partial z_{1}^{2}}\right|_{z_{1}=0}$};
\draw[<-](delta) to node {} (12+A1);
\draw[<-](12+A1) to node {} (12+2A1);
\draw[<-](12+2A1) to node {} (12+3A1);
\draw[<-](12+3A1) to node {} (12+4A1);
 \draw[<-,dashed](delta) to node {} (10+A1);
\draw[<-](10+A1) to node {} (10+2A1);
\draw[<-](10+2A1) to node {} (10+3A1);
\draw[<-](10+3A1) to node {} (10+4A1);
 \draw[<-,dashed](10+A1) to node {} (8+2A1);
\draw[<-](8+2A1) to node {} (8+3A1);
\draw[<-](8+3A1) to node {} (8+4A1);
 \draw[<-,dashed](8+2A1) to node {} (6+3A1);
 \draw[<-,dashed](6+3A1) to node {} (4+4A1);
\draw[<-](6+3A1) to node {} (6+4A1);
\end{tikzpicture} 
\end{center}
where $\varphi_{k,NA_{1}}\in J_{k,NA_{1};1}$. Except for the last line all forms are cusp forms. The functions at the beginning of each arrow are defined by application of $\otimes\phi_{0,1}$ to the function on the arrowhead. Here the quasi-pullbacks on the diagonal  are taken with respect to the coordinates (\ref{choice of coordinates for NA1}).
\end{prop}
\begin{proof}
The pullback-structure of the diagram follows from Proposition \ref{prop:singular theta series for 4A1} and Proposition \ref{prop:theta type construction for arbitrary root lattice}. If we rewrite the Fourier expansion of $\phi_{0,1}(\tau,z)$ given in Example \ref{eta function and Gritsenkos theta function}, part (c) according to Definition \ref{def: Jacobi forms}, we obtain that the hyperbolic norm of all indices belonging to non-vanishing Fourier coefficients is bounded from below by $-\frac{1}{2}$. Note that multiplication with $\eta(\tau)^{6}$ adds $\frac{1}{2}$ to this bound. If we take into account that $NA_{1}$ consists of $N$ perpendicular copies of $A_{1}$ we find that this proves the assertion on the holomorphicity and on cusp forms.  Part (d) of the same Example yields the quasi-pullback structure of the diagonal. 
\end{proof}
\noindent We consider another series of root lattices. For $N\geq 3$ the root system $D_{N}$ is described  by
\[\left(\langle\varepsilon_{2}+\varepsilon_{1},\varepsilon_{2}-\varepsilon_{1},\varepsilon_{3}-\varepsilon_{2}\dots,\varepsilon_{N}-\varepsilon_{N-1}\rangle_{\Z},(\cdot,\cdot)_{N}\right)\]
and $D_{2},D_{1}$ can be considered as $\left(\varepsilon_{2}+\varepsilon_{1},\varepsilon_{2}-\varepsilon_{1},(\cdot,\cdot)_{2}\right)$ or $\left(2\varepsilon_{1},(\cdot,\cdot)_{1}\right)$, respectively.  The first of the following two constructions also appeared in \cite{G1} but in different coordinates.
\begin{prop}\label{prop:singular Jacobi forms for DN}
 \begin{enumerate}[(a)]
\item  We have $\vartheta_{D_{N}}\in J_{N/2,D_{N},1}(v_{\eta}^{3N})$ and for $N=1,\dots,8$ there are Jacobi forms $\psi_{12-N,D_{N}}\in J_{12-N,D_{N};1}$ given by
\[\psi_{12-N,D_{N}}(\tau,\mathfrak{z}_{D_{N}})=\eta(\tau)^{24-3N}\vartheta_{D_{N}}(\tau,\mathfrak{z}_{D_{N}})\,.\]
For $N<8$ these functions define Jacobi cusp forms. These forms are obtained by iterative quasi-pullbacks of $\vartheta_{D_{8}}$, namely
\[\left.\frac{\partial \psi_{12-N,D_{N}}}{\partial z_{N}}\right|_{z_{N}=0}=2\pi i\,\psi_{12-N+1,D_{N-1}}\,.\]
\item   We have the following diagram of holomorphic Jacobi forms where all forms except for the first column are cups forms of the weight indicated by the index
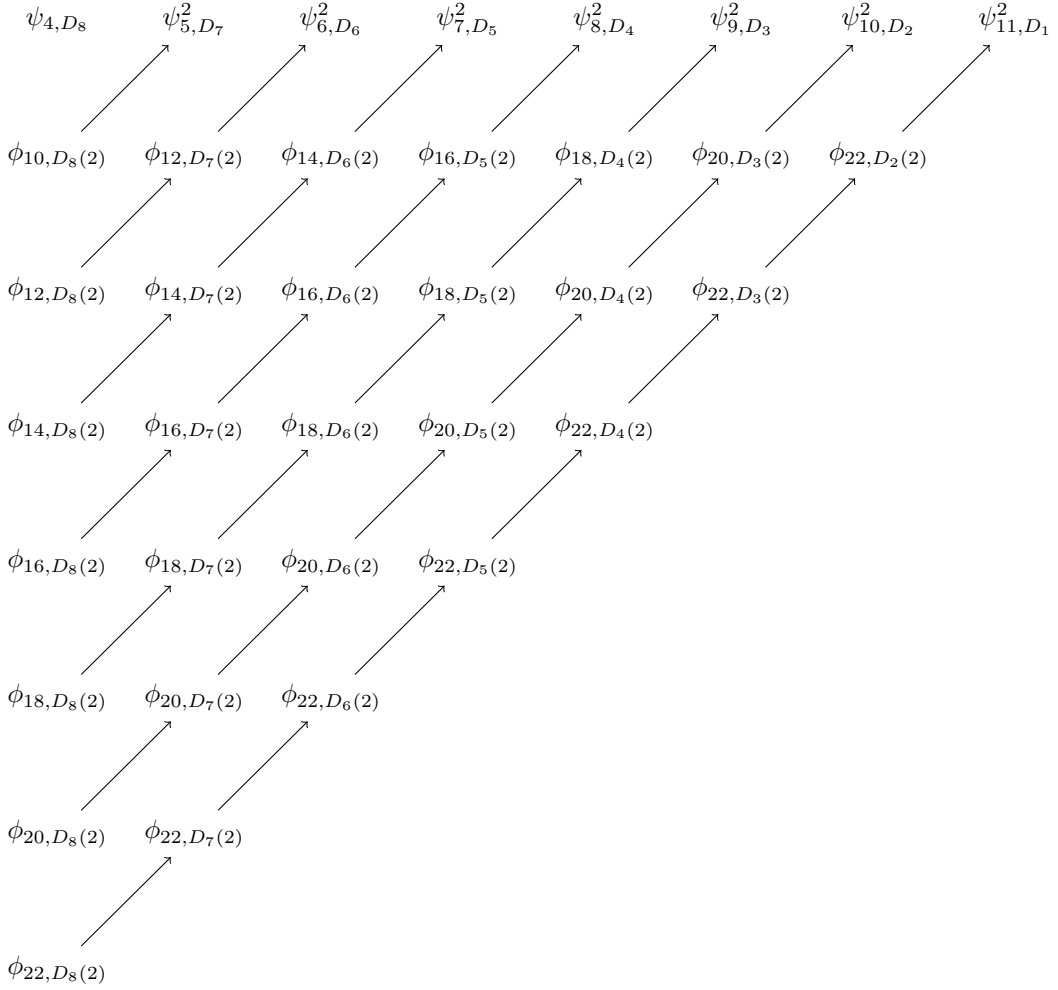
\begin{center}
 \begin{tikzpicture}[scale=0.36,node distance=2cm,auto]
  \node (4+D8) at (0cm,0cm) {$\psi_{4,D_{8}}$};
  \node (q5+D8) at (0cm,-5cm) {$\phi_{10,D_{8}(2)}$};
  \node (q6+D8) at (0cm,-10cm) {$\phi_{12,D_{8}(2)}$};
  \node (14+D8) at (0cm,-15cm) {$\phi_{14,D_{8}(2)}$};
  \node (16+D8) at (0cm,-20cm) {$\phi_{16,D_{8}(2)}$};
  \node (18+D8) at (0cm,-25cm) {$\phi_{18,D_{8}(2)}$};
  \node (20+D8) at (0cm,-30cm) {$\phi_{20,D_{8}(2)}$};
  \node (22+D8) at (0cm,-35cm) {$\phi_{22,D_{8}(2)}$};
  \node[] (q5+D7) at (5cm,0cm) {$\psi^{2}_{5,D_{7}}$};
  \node[] (q6+D7) at (5cm,-5cm) {$\phi_{12,D_{7}(2)}$};
  \node[] (14+D7) at (5cm,-10cm) {$\phi_{14,D_{7}(2)}$};
  \node[] (16+D7) at (5cm,-15cm) {$\phi_{16,D_{7}(2)}$};
  \node[] (18+D7) at (5cm,-20cm) {$\phi_{18,D_{7}(2)}$};
  \node[] (20+D7) at (5cm,-25cm) {$\phi_{20,D_{7}(2)}$};
  \node[] (22+D7) at (5cm,-30cm) {$\phi_{22,D_{7}(2)}$};
  \node[] (q6+D6) at (10cm,0cm) {$\psi^{2}_{6,D_{6}}$};
  \node[] (14+D6) at (10cm,-5cm) {$\phi_{14,D_{6}(2)}$};
  \node[] (16+D6) at (10cm,-10cm) {$\phi_{16,D_{6}(2)}$};
  \node[] (18+D6) at (10cm,-15cm) {$\phi_{18,D_{6}(2)}$};
  \node[] (20+D6) at (10cm,-20cm) {$\phi_{20,D_{6}(2)}$};
  \node[] (22+D6) at (10cm,-25cm) {$\phi_{22,D_{6}(2)}$};
  \node[] (q7+D5) at (15cm,0cm) {$\psi^{2}_{7,D_{5}}$};
  \node[] (16+D5) at (15cm,-5cm) {$\phi_{16,D_{5}(2)}$};
  \node[] (18+D5) at (15cm,-10cm) {$\phi_{18,D_{5}(2)}$};
  \node[] (20+D5) at (15cm,-15cm) {$\phi_{20,D_{5}(2)}$};
  \node[] (22+D5) at (15cm,-20cm) {$\phi_{22,D_{5}(2)}$};
  \node[] (q8+D4) at (20cm,-0cm) {$\psi^{2}_{8,D_{4}}$};
  \node[] (18+D4) at (20cm,-5cm) {$\phi_{18,D_{4}(2)}$};
  \node[] (20+D4) at (20cm,-10cm) {$\phi_{20,D_{4}(2)}$};
  \node[] (22+D4) at (20cm,-15cm) {$\phi_{22,D_{4}(2)}$};
  \node[] (q9+D3) at (25cm,0cm) {$\psi^{2}_{9,D_{3}}$};
  \node[] (20+D3) at (25cm,-5cm) {$\phi_{20,D_{3}(2)}$};
  \node[] (22+D3) at (25cm,-10cm) {$\phi_{22,D_{3}(2)}$};
  \node[] (q10+D2) at (30cm,-0cm) {$\psi^{2}_{10,D_{2}}$};
  \node[] (22+D2) at (30cm,-5cm) {$\phi_{22,D_{2}(2)}$};
  \node[] (q11+D1) at (35cm,-0cm) {$\psi^{2}_{11,D_{1}}$};
  \draw[<-] (q11+D1) to node {} (22+D2);
  \draw[<-] (q10+D2) to node {} (20+D3);
  \draw[<-] (22+D2) to node {} (22+D3);
  \draw[<-] (q9+D3) to node {} (18+D4);
 \draw[<-] (20+D3) to node {} (20+D4);
 \draw[<-] (22+D3) to node {} (22+D4);
  \draw[<-] (q8+D4) to node {} (16+D5);
  \draw[<-] (18+D4) to node {} (18+D5);
  \draw[<-] (20+D4) to node {} (20+D5);
  \draw[<-] (22+D4) to node {} (22+D5);
  \draw[<-] (q7+D5) to node {} (14+D6);
  \draw[<-] (16+D5) to node {} (16+D6);
  \draw[<-] (18+D5) to node {} (18+D6);
  \draw[<-] (20+D5) to node {} (20+D6);
  \draw[<-] (22+D5) to node {} (22+D6);
  \draw[<-] (q6+D6) to node {} (q6+D7);
  \draw[<-] (14+D6) to node {} (14+D7);
  \draw[<-] (16+D6) to node {} (16+D7);
  \draw[<-] (18+D6) to node {} (18+D7);
  \draw[<-] (20+D6) to node {} (20+D7);
  \draw[<-] (22+D6) to node {} (22+D7);
  \draw[<-] (q5+D7) to node {} (q5+D8);
  \draw[<-] (q6+D7) to node {} (q6+D8);
  \draw[<-] (14+D7) to node {} (14+D8);
  \draw[<-] (16+D7) to node {} (16+D8);
  \draw[<-] (18+D7) to node {} (18+D8);
  \draw[<-] (20+D7) to node {} (20+D8);
  \draw[<-] (22+D7) to node {} (22+D8);
\end{tikzpicture} 
\end{center}
where the functions at the beginning of each arrow are defined by application of $\prescript{}{D_{N-1}(2)}{\prod}^{D_{N}(2)}(\cdot,\phi_{0,1})$ to the function on the arrowhead.
 More precisely we have 
\[\phi_{k,D_{N}(2)}\in J_{k,D_{N}(2),1}\]
and $\phi_{k,D_{N}(2)}$ is a cusp form if and only if $N<8$. All the functions except for the first line have the property
\[\prescript{}{D_{N-1}(2)}{\prod}^{D_{N}(2)}(\phi_{k,D_{N-1}(2)},\phi_{0,1})\to\phi_{k,D_{N-1}(2)}\text{  where }N\geq 2\]
regarding the convention $\phi_{26-2N,D_{N-1}(2)}:=\psi^{2}_{13-N,D_{N-1}}$
 \end{enumerate}
\end{prop}
\noindent The arithmetic Lifting $\operatorname{A-Lift}(\psi_{11,D_{1}})$ is the cusp form of weight 11 for the Siegel paramodular group of level 2.
\begin{proof}
 \begin{enumerate}[(a)]
  \item  We first note that our model for $D_{N}$ is compatible with the requirements of Proposition \ref{prop:theta type construction for arbitrary root lattice}. Hence the statements on the modular behaviour and the identification of the character follow directly as well as the property under quasi-pullbacks. It remains to investigate the holomorphicity. Since $D_{N}\subseteq \R^{N}$ is a full lattice  we conclude from part (a) of Proposition \ref{prop:theta type construction for arbitrary root lattice} that $\vartheta_{D_{N}}$ is indeed holomorphic at infinity. Moreover for $N<8$ the forms $\psi_{12-N,D_{N}}$ are cuspidal because they equal the product of a holomorphic Jacobi form of singular weight and a cusp form of index 0. 
\item From part (a) and the iterative construction we can easily derive that $\phi_{k,D_{N}(2)}$ is a holomorphic Jacobi form whose Fourier coefficients $f(n,l)$ are parametrized by all pairs $(n,l)\in \N\times D_{N}(2)^{\vee}$ such that $2n-(l,l)\geq \frac{8-N}{2}$. Now Proposition \ref{prop:theta type construction for arbitrary root lattice} completes the proof.
 \end{enumerate}
\end{proof}
\noindent We finish this section considering root lattices of type $A_{N}$. Here we use the model
\[\left(\langle\varepsilon_{2}-\varepsilon_{1},\dots,\varepsilon_{N+1}-\varepsilon_{N}\rangle_{\Z},(\cdot,\cdot)_{N+1}\right)\]
for $A_{N}$ where $N\geq 1$. Note that in this case $A_{N}\subseteq \R^{N+1}$. Hence $A_{N}$ is not a full lattice according to our realization.
\begin{prop}
\begin{enumerate}[(a)]
 \item Let $N\in\{2,\dots 7\}$. We define 
\[\vartheta_{A_{N}}^{(1)}\,:\,\mathbb{H}\times(A_{N}\otimes\C)\to\C\quad,\quad \vartheta_{A_{N}}^{(1)}(\tau,\mathfrak{z}):=\eta(\tau)^{-1} \vartheta_{A_{N}}(\tau,\mathfrak{z})\,.\]
Then $\vartheta_{A_{N}}^{(1)}\in J_{N/2,A_{N};1}(v_{\eta}^{3N+2})$ and there are Jacobi forms $\psi_{11-N,A_{N}}\in J_{11-N,A_{N};1}$ given by
\[\psi_{11-N,A_{N}}(\tau,\mathfrak{z}_{A_{N}})=\eta(\tau)^{21-3N}\vartheta_{A_{N}}(\tau,\mathfrak{z}_{A_{N}})\,.\]
For $N<7$ these functions define Jacobi cusp forms. These forms are obtained by iterative quasi-pullbacks of $\psi_{A_{7}}$, namely
\[\left.\frac{\partial \psi_{11-N,A_{N}}}{\partial z_{N}}\right|_{z_{N}=0}=2\pi i\,\psi_{11-N+1,A_{N-1}}\,.\]
\item   We have the following diagram of holomorphic Jacobi forms where all forms except for the first column are cups forms of the weight indicated by the index
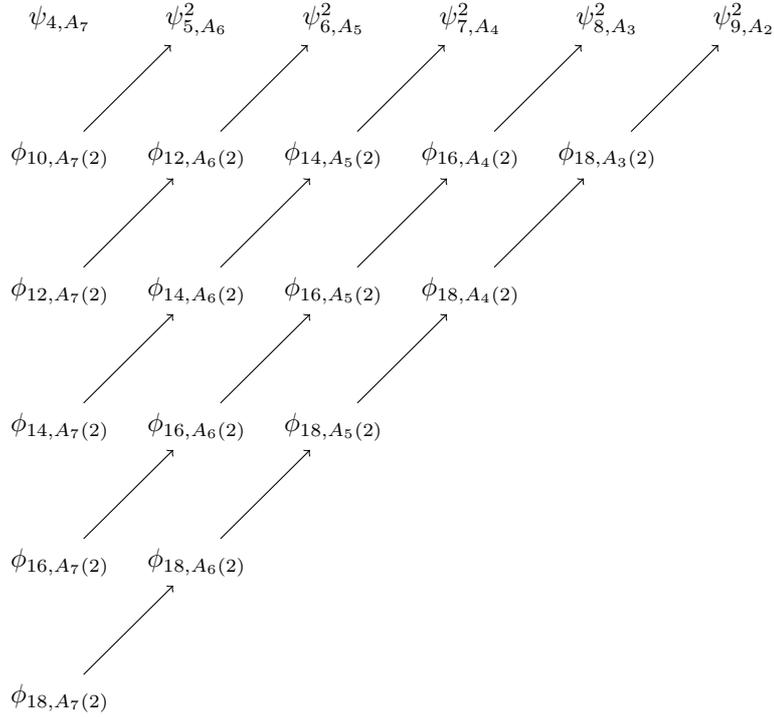
\begin{center}
 \begin{tikzpicture}[scale=0.36,node distance=2cm,auto]
  \node (4+D8) at (0cm,0cm) {$\psi_{4,A_{7}}$};
  \node (q5+D8) at (0cm,-5cm) {$\phi_{10,A_{7}(2)}$};
  \node (q6+D8) at (0cm,-10cm) {$\phi_{12,A_{7}(2)}$};
  \node (14+D8) at (0cm,-15cm) {$\phi_{14,A_{7}(2)}$};
  \node (16+D8) at (0cm,-20cm) {$\phi_{16,A_{7}(2)}$};
  \node (18+D8) at (0cm,-25cm) {$\phi_{18,A_{7}(2)}$};
  \node[] (q5+D7) at (5cm,0cm) {$\psi^{2}_{5,A_{6}}$};
  \node[] (q6+D7) at (5cm,-5cm) {$\phi_{12,A_{6}(2)}$};
  \node[] (14+D7) at (5cm,-10cm) {$\phi_{14,A_{6}(2)}$};
  \node[] (16+D7) at (5cm,-15cm) {$\phi_{16,A_{6}(2)}$};
  \node[] (18+D7) at (5cm,-20cm) {$\phi_{18,A_{6}(2)}$};
  \node[] (q6+D6) at (10cm,0cm) {$\psi^{2}_{6,A_{5}}$};
  \node[] (14+D6) at (10cm,-5cm) {$\phi_{14,A_{5}(2)}$};
  \node[] (16+D6) at (10cm,-10cm) {$\phi_{16,A_{5}(2)}$};
  \node[] (18+D6) at (10cm,-15cm) {$\phi_{18,A_{5}(2)}$};
  \node[] (q7+D5) at (15cm,0cm) {$\psi^{2}_{7,A_{4}}$};
  \node[] (16+D5) at (15cm,-5cm) {$\phi_{16,A_{4}(2)}$};
  \node[] (18+D5) at (15cm,-10cm) {$\phi_{18,A_{4}(2)}$};
  \node[] (q8+D4) at (20cm,-0cm) {$\psi^{2}_{8,A_{3}}$};
  \node[] (18+D4) at (20cm,-5cm) {$\phi_{18,A_{3}(2)}$};
  \node[] (q9+D3) at (25cm,0cm) {$\psi^{2}_{9,A_{2}}$};
  \draw[<-] (q9+D3) to node {} (18+D4);
  \draw[<-] (q8+D4) to node {} (16+D5);
  \draw[<-] (18+D4) to node {} (18+D5);
  \draw[<-] (q7+D5) to node {} (14+D6);
  \draw[<-] (16+D5) to node {} (16+D6);
  \draw[<-] (18+D5) to node {} (18+D6);
  \draw[<-] (q6+D6) to node {} (q6+D7);
  \draw[<-] (14+D6) to node {} (14+D7);
  \draw[<-] (16+D6) to node {} (16+D7);
  \draw[<-] (18+D6) to node {} (18+D7);
  \draw[<-] (q5+D7) to node {} (q5+D8);
  \draw[<-] (q6+D7) to node {} (q6+D8);
  \draw[<-] (14+D7) to node {} (14+D8);
  \draw[<-] (16+D7) to node {} (16+D8);
  \draw[<-] (18+D7) to node {} (18+D8);
\end{tikzpicture} 
\end{center}
where the functions at the beginning of each arrow are defined by application of $\prescript{}{A_{N-1}(2)}{\prod}^{A_{N}(2)}(\cdot,\phi_{0,1})$ to the function on the arrowhead.
 More precisely we have 
\[\phi_{k,A_{N}(2)}\in J_{k,A_{N}(2),1}\]
and $\phi_{k,A_{N}(2)}$ is a cusp form if and only if $N<7$. All the functions except for the first line have the property
\[\prescript{}{A_{N-1}(2)}{\prod}^{A_{N}(2)}(\phi_{k,A_{N-1}(2)},\phi_{0,1})\to\phi_{k,A_{N-1}(2)}\text{  where }N\geq 3\]
regarding the convention $\phi_{24-2N,A_{N-1}(2)}:=\psi^{2}_{12-N,A_{N-1}}$
\end{enumerate}
\end{prop}
\begin{proof}
\begin{enumerate}[(a)]
 \item  We start with the treatment of $\vartheta_{A_{N}}^{(1)}$. As before the transformation behaviour under modular substitutions and the shape of the character follow directly from Proposition \ref{prop:theta type construction for arbitrary root lattice}. The only thing which remains to show is the holomorphicity. In this case the proof of this fact is less obvious than in the previous cases since the dimension of the ambient vector space of $A_{N}$ is strictly greater than $N$.  We start with the Fourier expansion of the function $\vartheta_{A_{2}}(\tau,\mathfrak{z}_{A_{2}})$. Each summand is of	 the shape
\[q^{\frac{n_{1}^{2}+n_{2}^{2}+n_{3}^{2}}{8}}\exp(\pi i(-n_{1}z_{1}+n_{2}z_{2}+n_{3}(z_{1}-z_{2})))\]
where $n_{1}\equiv n_{2} \equiv n_{3}\equiv 1 \bmod 2 $\,. Since $A_{2}^{\vee}\subseteq \frac{1}{6} A_{2}$ the last term can be rewritten as
\[q^{\frac{n_{1}^{2}+n_{2}^{2}+n_{3}^{2}}{8}}\exp(2\pi i(v,\mathfrak{z}_{A_{2}}))\]
where
\[v:=\frac{1}{6}\left[n_{1}(2\varepsilon_{1}-\varepsilon_{2}-\varepsilon_{3})+n_{2}(2\varepsilon_{3}-\varepsilon_{1}-\varepsilon_{2})+n_{3}(2\varepsilon_{2}-\varepsilon_{1}-\varepsilon_{3})\right]\,.\]
The hyperbolic norm of the last expression equals
\[\begin{aligned}
   &2\frac{n_{1}^{2}+n_{2}^{2}+n_{3}^{2}}{8}-(v,v)&\\
&\quad\quad\quad\quad\quad\quad=\frac{n_{1}^{2}+n_{2}^{2}+n_{3}^{2}}{4}-\frac{n_{1}^{2}+n_{2}^{2}+n_{3}^{2}}{6}+\frac{n_{1}n_{2}+n_{1}n_{3}+n_{2}n_{3}}{6}&\\
  &\quad\quad\quad\quad\quad\quad=\frac{(n_{1}+n_{2}+n_{3})^{2}}{12}\geq	 \frac{1}{12}&
  \end{aligned}\] 
since the parity of $n_{1},n_{2},n_{3}$ is odd. This shows that $\vartheta_{A_{2}}^{(1)}$ is indeed a holomorphic Jacobi form since
\[\eta(\tau)^{-1}=q^{-1/24}(1+q(\dots))\,.\]
Now let $N>2$. We decompose $\mathfrak{z}\in A_{N}\otimes\C$ as 
\[\mathfrak{z}=\mathfrak{z}_{A_{N-1}}\oplus \mathfrak{z}_{K}\quad,\quad K=(A_{N-1})^{\perp}_{A_{N}}\,.\]
With respect to this basis $\vartheta_{A_{N}}^{(1)}$ decomposes as
\[\vartheta_{A_{N}}^{(1)}(\tau,\mathfrak{z})=\vartheta_{A_{N-1}}^{(1)}(\tau,\mathfrak{z}_{A_{N-1}})\cdot \vartheta_{K}(\tau,\mathfrak{z}_{K})\]
where $K$ is realized as a lattice in $\R$. Now an induction and the holomorphicity of $\vartheta_{K}$  show that $\vartheta_{A_{N}}^{(1)}$ is a holomorphic Jacobi form. The statements on $\psi_{11-N,A_{N}}$ are obvious.
\item The proof of this assertion is completely analogue to the proof of Proposition \ref{prop:singular Jacobi forms for DN} part (b).
\end{enumerate}
\end{proof}
\noindent The length of the towers can be arbitrarily increased if we multiply them by powers of $\Delta_{12}$ to preserve the holomorphicity.
\newpage
\section{Additional symmetries for the lattice $D_{4}$}
\noindent In this section we use theta type Jacobi forms to explain the existence of a particular modular form for the lattice $D_{4}$. For any $r\in L_{2}$ we define the reflection at the hyperplane $r^{\perp}$ as
\[\sigma_{r}\,:\,V\to V\quad,\quad \sigma_{r}(v)=v-2\frac{(r,v)}{(r,r)}r\]
 Let $r\in L_{2}\otimes\Q$ such that $(r,r)<0$. The rational quadratic divisor with respect to $r$ is given as
 \[\mathcal{D}_{r}(L_{2}):=\{[\mathcal{Z}]\in \mathcal{D}(L_{2})\,|\,(\mathcal{Z},r)=0\}\,.\]
 In particular
 \[[\mathcal{Z}]\in\mathcal{D}_{r}(L_{2})\iff \sigma_{r}([\mathcal{Z}])=[\mathcal{Z}]\,.\]
 Note that \(\sigma_{r}\in\operatorname{O}(L_{2}\otimes\Q)^{+}\iff (r,r)<0\,.\)
We now state a variant of Borcherds multiplicative lifting which can be found in \cite{G2}.
\begin{thm}[Borcherds, Gritsenko]\label{theorem:borcherds lifting for weak jacob forms}
 Let $\varphi$ be a weakly holomorphic Jacobi form with Fourier coefficients $f(n,l)$.
 Assume that $f(n,l)\in \Z$ if $2n-(l,l)\leq 0$\,. Then there exists a modular form $F_{\varphi}$ of weight $\frac{1}{2}f(0,0)$ 
 with respect to $\widetilde{\operatorname{O}}(L_{2})^{+}$ satisfying
 \[\begin{aligned}
    &&&\operatorname{div}(F_{\varphi})&&=&&\sum_{\substack{h\in L_{2}^{\vee}\text{ primitive}\\(h,h)<0}}{\alpha_{h}\mathcal{D}_{h}(L_{2})}&\\
    &\text{where }&&\qquad\alpha_{h}&&=&&\sum_{\substack{d>0\\(h,h)=2n-(l,l)\\h\equiv l \bmod L_{2}}}{f(d^{2}n,dl)}\,.&
   \end{aligned}\]
\end{thm}
\noindent For the details of the construction of $F_{\varphi}$ we refer to \cite{G2} and \cite{Bo}. 
The root lattices of type $D_{N}$ have been defined in the last section. Using the same coordinates as before let $\mathfrak{z}\in D_{N}\otimes \C$. In general the only symmetries of the Coxeter diagram of $D_{N}$ are given by the permutation of $z_{1}:=(\mathfrak{z},\varepsilon_{1})$ and $z_{2}:=(\mathfrak{z},\varepsilon_{2})$:
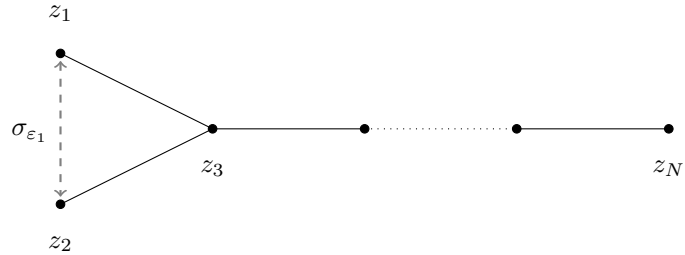
\vspace{5mm}

\begin{figure}[h]
\begin{center}
 \begin{tikzpicture}
 \node[below] at (0,-1.3) {$z_{2}$};
 \node[above] at (0,1.3) {$z_{1}$};
 \node[above] at (-0.4,-0.3) {$\sigma_{\varepsilon_{1}}$};
 \node[below] at (2,-0.3) {$z_{3}$};
 \node[below] at (8,-0.3) {$z_{N}$};
 \filldraw[black] (0,1) circle (1.6pt);
\filldraw[black] (0,-1) circle (1.6pt);
 \filldraw[black] (2,0) circle (1.6pt);
 \filldraw[black] (4,0) circle (1.6pt);
 \filldraw[black] (6,0) circle (1.6pt);
 \filldraw[black] (8,0) circle (1.6pt);
 \draw[](0,1)--(2,0);
 \draw[thick,gray,dashed,<->](0,0.9)--(0,-0.9);
 \draw[](0,-1)--(2,0);
 \draw[](2,0)--(4,0);
 \draw[dotted](4,0)--(6,0);
 \draw[](6,0)--(8,0);
\end{tikzpicture}
\end{center}
\end{figure}
\vspace{5mm}

\noindent If $N=4$ the diagram has additional symmetries:
\begin{figure}[ht]
\begin{center}
 \begin{tikzpicture}
 \node[below] at (5:2.3) {$z_{4}$};
 \node[above] at (124:2.1) {$z_{1}$};
 \node[below] at (0:0.3) {$z_{3}$};
 \node[below] at (240:2.1) {$z_{2}$};
 \filldraw[black] (0:0) circle (1.6pt);
\filldraw[black] (0:2) circle (1.6pt);
 \filldraw[black] (120:2) circle (1.6pt);
 \filldraw[black] (240:2) circle (1.6pt);
 \draw[](0:0)--(120:2);
\draw[](0:0)--(240:2);
\draw[](0:0)--(0:2);
 \draw[<->,thick,dashed,gray](4:2) arc (4:116:2);
 \draw[<->,thick,dashed,gray](124:2) arc (124:236:2);
 \draw[<->,thick, dashed,gray](244:2) arc (244:356:2);
\node[above] at (60:2) {$\sigma_{v}$};
\node[above] at (180:2.4) {$\sigma_{\varepsilon_{1}}$};
\node[below] at (300:2.2) {$\sigma_{\varepsilon_{1}}\sigma_{v}\sigma_{\varepsilon_{1}}$};
 \node[below] at (5:7.3) {$\displaystyle v=\frac{\varepsilon_{1}+\varepsilon_{2}+\varepsilon_{3}-\varepsilon_{4}}{2}\in D_{4}^{\vee}\,.$};
\end{tikzpicture}
\end{center}
\end{figure}

\noindent In this case we can define two additional singular Jacobi forms
\[\vartheta_{D_{4}}^{(2)}:=\left.\vartheta_{D_{4}}\right|_{2,1}\sigma_{\varepsilon_{1}}\sigma_{v}\sigma_{\varepsilon_{1}}\quad,\quad \vartheta_{D_{4}}^{(3)}:=\left.\vartheta_{D_{4}}\right|_{2,1}\sigma_{v}\,.\]
If we use our standard coordinates  we obtain:
\[\begin{aligned}
   \vartheta_{D_{4}}(\tau,\mathfrak{z}_{D_{4}})&=&\vartheta(\tau,z_{1}-z_{2})\vartheta(\tau,z_{1}+z_{2}-z_{3})\vartheta(\tau,z_{3}-z_{4})\vartheta(\tau,z_{4})\\
\vartheta_{D_{4}}^{(2)}(\tau,\mathfrak{z}_{D_{4}})&=&\vartheta(\tau,z_{2})\vartheta(\tau,z_{3}-z_{2})\vartheta(\tau,z_{1}-z_{3}+z_{4})\vartheta(\tau,z_{1}-z_{4})\\
\vartheta_{D_{4}}^{(3)}(\tau,\mathfrak{z}_{D_{4}})&=&\vartheta(\tau,z_{1})\vartheta(\tau,z_{1}-z_{3})\vartheta(\tau,z_{4}-z_{2})\vartheta(\tau,-z_{2}+z_{3}-z_{4})\,.
  \end{aligned}\] 
\noindent For any $0\neq l\in L_{2}$ we define $\operatorname{div}(l)\in\N$ to be the generator of the ideal $\left((l,h)\,|\,h\in L_{2}\right)_{\Z}$ and we put $l^{\ast}:=l/\operatorname{div}(l)$. A vector $l\in L_{2}$ is called primitive if
\[\forall k\in\Z,h\in L_{2}\,:\,(l=kh\implies k=\pm 1)\,.\]
The following Proposition is very useful if one likes to determine the orbits of the divisor of a multiplicative lifting under the orthogonal group. A proof can be found in \cite{GHS}.
\begin{prop}[Eichler criterion]\label{Eichler criterion}  
 If $r,s\in L_{2}$ are primitive, $(s,s)=(r,r)$ and $r^{\ast}\equiv s^{\ast}\bmod L_{2}$ then there exists a $g\in \widetilde{\operatorname{SO}}(L_{2})^{+}$ such that $g(r)=s$.
\end{prop}
\noindent The next Theorem can be found in \cite{G1}.
\begin{thm}[Gritsenko '10]\label{Lifting construction of reflective modular forms for DN tower}
  Let $N\in\{1,\dots,8\}$. There exists a modular form $G_{12-N}^{D_{N}}\in \mathcal{M}_{12-N}(\Gamma,\chi)$ where
\[\Gamma=\begin{cases}
          \operatorname{O}(L_{2}(D_{N}))^{+}&\text{ if }N\neq 4\\
	  \left<\widetilde{\operatorname{O}}(L_{2}(D_{4}))^{+},\sigma_{\varepsilon_{1}}\right>&\text{ if } N=4
         \end{cases}\,.\]
In the last case $\Gamma\leq \operatorname{O}(L_{2}(D_{4}))^{+}$ is a subgroup of index three and the character is defined as
\[\chi:\Gamma\xrightarrow[]{}\Gamma/\widetilde{\operatorname{O}}(L_{2}(D_{N}))^{+}\xrightarrow[]{\operatorname{sign}}\{\pm 1\}\,.\]
 The divisor of $G_{12-N}^{D_{N}}$ is given by 
\[\sum_{\substack{\pm r \in L_{2}(D_{N})\\(r,r)=-4\\\operatorname{div}(r)=2}}{\mathcal{D}_{r}(L_{2}(D_{N}))}=\widetilde{\operatorname{O}}(L_{2}(D_{N}))^{+}.\mathcal{D}_{\varepsilon_{1}}(L_{2}(D_{N}))\]
if $N\neq 4$ and for $D_{4}$ it is given by the orbit
\[\widetilde{\operatorname{O}}(L_{2}(D_{N}))^{+}.\mathcal{D}_{\varepsilon_{1}}(L_{2}(D_{N}))\,.\]
 If $N\neq 8$ then $G_{12-N}^{D_{N}}$ is a cusp form. 
\end{thm}
\noindent The proof of this Theorem uses Theorem  \ref{theorem:borcherds lifting for weak jacob forms}. The case $N=4$ plays a particular role in this tower because the maximal modular group and the divisor of $G_{8}^{D_{4}}$ are smaller. In the next Theorem we use theta type Jacobi forms to construct a modular form of weight 24 which can be considered as the correct replacement for $G_{8}^{D_{4}}$. 
\begin{thm}
There is a cusp form $\Delta_{24}^{D_{4}}\in \mathcal{S}_{24}(\operatorname{O}(L_{2}(D_{4}))^{+},\chi)$ 
with a binary character
\[\chi:\Gamma\xrightarrow[]{}\Gamma/\widetilde{\operatorname{O}}(L_{2}(D_{4}))^{+}\xrightarrow[]{\operatorname{sign}}\{\pm 1\}\quad\text{where}\quad \Gamma=\operatorname{O}(L_{2}(D_{4}))^{+}\,.\]
 Its divisor is equal to 
\[\sum_{\substack{\pm r \in L_{2}(D_{4})\\(r,r)=-4\\\operatorname{div}(r)=2}}{\mathcal{D}_{r}(L_{2}(D_{4}))}\,.\]
\end{thm}
\noindent This function appeared in \cite{FH} first and later in \cite{K2} where the author determined the graded ring of modular forms for the Hurwitz order.
\begin{proof}
 For the proof we will first verify that 
\[\operatorname{A-Lift}(\eta(\tau)^{12}\vartheta_{D_{4}}),\operatorname{A-Lift}(\eta(\tau)^{12}\vartheta_{D_{4}}^{(2)}),\operatorname{A-Lift}(\eta(\tau)^{12}\vartheta_{D_{4}}^{(3)})\] are multiplicative liftings as in Theorem \ref{theorem:borcherds lifting for weak jacob forms}. The discriminant group of $D_{4}$ has the four classes
\[D_{4},\varepsilon_{1}+D_{4},v+D_{4},v+\varepsilon_{1}+D_{4}\]
where $v$ was defined in the diagram above. The group $\operatorname{O}(L_{2}(D_{4}))^{+}$ is generated by $\widetilde{\operatorname{O}}(L_{2}(D_{4}))^{+}$ and the natural embeddings of $\sigma_{\varepsilon_{1}}$ and $\sigma_{v}$ into $\operatorname{O}(L_{2}(D_{4}))^{+}$. 
We start by investigating $\psi_{8,D_{4}}(\tau,\mathfrak{z})=\eta(\tau)^{12}\vartheta_{D_{4}}(\tau,\mathfrak{z})$. From the Fourier-Jacobi criterion for automorphic products, compare \cite{G2}, Corollary 3.3 we deduce that 
\[ \varphi_{0,D_{4}}=-\,\frac{2^{-1}\psi_{8,D_{4}}|{\operatorname{T}_{-}(2)}}{\psi_{8,D_{4}}}\]
where $\operatorname{T}_{-}(2)$ is the Hecke operator defined in \cite{G}, if $\operatorname{A-Lift}(\psi_{8,D_{4}})=F_{\varphi_{0,D_{4}}}$. A direct computation yields
\[\varphi_{0,D_{4}}(\tau,\mathfrak{z})=16+s_{1}+s_{2}+s_{3}+s_{4}+s_{1}^{-1}+s_{2}^{-1}+s_{3}^{-1}+s_{4}^{-1}+q(\dots)\]
where
\[s_{j}=\exp(2\pi i(\mathfrak{z},\varepsilon_{j}))\,.\]
Hence $\varphi_{0,D_{4}}$ is a weak Jacobi form which satisfies the preliminaries of Theorem \ref{theorem:borcherds lifting for weak jacob forms}. If $L=D_{N}$ the hyperbolic norm in the Fourier exansion of a weak Jacobi form is bounded from below by the quantity
\[-\max_{l\in L^{\vee}/L}\left(\min_{v\in l+L}{(v,v)}\right)=-\frac{N}{4}\,.\] 
Now assume there is a pair $(n,l)\in\N\times L^{\vee}$ such that $2n-(l,l)<0$. Then one observes
\[2n-(l,l)\in \{-N/4,-1\}\,.\]
Since the lattice $D_{N}^{\vee}$ represents all such values under $-(\cdot,\cdot)_{D_{N}}$ we know that  all the information about the divisor is completely contained in the $q^{0}$ part of $\varphi_{0,D_{4}}$ since the Fourier coefficients of a Jacobi form depend only on the hyperbolic norm and the class of $l$. Now Theorem \ref{theorem:borcherds lifting for weak jacob forms} and Proposition \ref{Eichler criterion} yield that the divisor of $F_{\varphi_{0,D_{4}}}$ is exactly
\[\widetilde{\operatorname{O}}(L_{2}(D_{4}))^{+}.\mathcal{D}_{\varepsilon_{1}}(L_{2}(D_{4}))\,.\]
By constrcution this divisor is contained in the divisor of $\operatorname{A-Lift}(\eta(\tau)^{12}\vartheta_{D_{4}})$. By Koecher's principle we can divide the last function by $F_{\varphi_{0,D_{4}}}$ and conclude that they coincide up to a multiple in $\C^{\times}$ and a second application of the Fourier-Jacobi criterion for automorphic products yields that this constant is actually one in accordance with Theorem \ref{Lifting construction of reflective modular forms for DN tower}. So far we have just repeated the proof of Theorem \ref{Lifting construction of reflective modular forms for DN tower} for the case $N=4$. Now the same construction can be applied to  $\operatorname{A-Lift}(\eta(\tau)^{12}\vartheta_{D_{4}}^{(2)}),\operatorname{A-Lift}(\eta(\tau)^{12}\vartheta_{D_{4}}^{(3)})$ which yields again modular forms of weight $8$ where the divisor is determined by the orbit of $v$ or $\varepsilon_{1}+v$, respectively. We consider the product 
\[\psi(\tau,\mathfrak{z}):=\eta(\tau)^{36}\vartheta_{D_{4}}(\tau,\mathfrak{z})\vartheta_{D_{4}}^{(2)}(\tau,\mathfrak{z})\vartheta_{D_{4}}^{(3)}(\tau,\mathfrak{z})\,.\]
Then $\psi$ has the representation 
\[\psi(\tau,\mathfrak{z})=\eta(\tau)^{36}\cdot\prod_{j=1}^{4}{\vartheta(\tau,(\mathfrak{z},\varepsilon_{j}))\vartheta(\tau,(\mathfrak{z},\sigma_{v}(\varepsilon_{j})))\vartheta(\tau,(\mathfrak{z},\sigma_{\varepsilon_{1}}\sigma_{v}\sigma_{\varepsilon_{1}}(\varepsilon_{j})))}\,.\]
From this representation and the above representation in coordinates we deduce
\[\psi|_{24,3}\sigma_{\varepsilon_{1}}=-\psi\quad,\quad \psi|_{24,3}\sigma_{v}=-\psi\,.\]
Since these two elements generate the group of the graph automorphisms of the Coxeter diagram of $D_{4}$ which is isomorphic to the symmetric group in three letters $\mathcal{S}_{3}$ we see that $\psi$ is $\mathcal{S}_{3}$-modular with respect to the $\operatorname{sign}$-character. If we define $\Delta_{24}^{D_{4}}$ as the product  
\[\operatorname{A-Lift}(\eta(\tau)^{12}\vartheta_{D_{4}})\operatorname{A-Lift}(\eta(\tau)^{12}\vartheta_{D_{4}}^{(2)})\operatorname{A-Lift}(\eta(\tau)^{12}\vartheta_{D_{4}}^{(3)})\]
the definition of $\operatorname{A-Lift}$ yields a cusp form whose divisor is the composition of the divisors of the factors which is $\operatorname{O}(L_{2}(D_{4}))^{+}$-modular with respect to a binary character induced by the sign-character.
\end{proof}
\section*{Acknowledgements}
\noindent I would like to thank Valery Gritsenko for many valuable discussions and for his help. 
%
%
%
%

\bibliography{/home/woitalla/Dissertation/Bibliography/bibliography}

\begin{thebibliography}{10}

\bibitem{Ap}
T.M. Apostol.
\newblock {\em {Modular Functions and Dirichlet Series in Number Theory}},
  volume~21.
\newblock Springer-Verlag New York, Berlin, Heidelberg, 1976.

\bibitem{Bo}
R.E. Borcherds.
\newblock {Automorphic forms with singularities on Grassmannians}.
\newblock {\em Invent. Math.}, 123:491--562, 1998.

\bibitem{BorJ}
A.~Borel and L.~Ji.
\newblock {\em {Compactifications of symmetric and locally symmetric spaces}}.
\newblock Mathematics: Theory \& Applications. Birkhäuser, Boston, 2006.

\bibitem{Br}
J.~H. Bruinier.
\newblock {\em {Borcherds Products on $\operatorname{O}(2,l)$ and Chern classes
  of Heegner Divisors}}, volume~1.
\newblock Springer-Verlag, Berlin, Heidelberg, 2002.

\bibitem{FB}
R.~Busam and E.~Freitag.
\newblock {\em {Funktionentheorie}}, volume~1.
\newblock Springer-Verlag, Berlin, Heidelberg, 1993.

\bibitem{CG}
F.~Clery and V.~Gritsenko.
\newblock {Modular forms of orthogonal type and Jacobi Theta-series}.
\newblock {\em arXiv:1106.4733 [math.AG]}.

\bibitem{EZ}
M.~Eichler and D.~Zagier.
\newblock {\em {The Theory of Jacobi Forms}}.
\newblock Birkhäuser, Boston, Basel, Stuttgart, 1985.

\bibitem{FH}
E.~Freitag and C.F. Hermann.
\newblock {Some modular varieties of low dimension}.
\newblock {\em Advances in Mathematics}, 152:203--287, 2000.

\bibitem{G2}
V.~Gritsenko.
\newblock {24 Faces of the Borcherds Modular Form $\Phi_{12}$}.
\newblock {\em ArXiv: 1203.6503 [math.AG]}.
\newblock 14 pp.

\bibitem{G1}
V.~Gritsenko.
\newblock {Reflective modular forms in Algebraic geometry}.
\newblock {\em ArXiv: 1005.3753 [math.AG]}.
\newblock 28 pp.

\bibitem{G}
V.~Gritsenko.
\newblock {Modular forms and moduli spaces of abelian and $K3$ surfaces}.
\newblock {\em Mathematica Gottingensis}, 26, 1993.

\bibitem{GHS}
V.~Gritsenko, K.~Hulek, and G.K. Sankaran.
\newblock {Abelianisation of orthogonal groups and the fundamental group of
  modular varieties}.
\newblock {\em ArXiv: 0810.1614 [math.AG]}.
\newblock 21 pp.

\bibitem{GN}
V.~Gritsenko and V.~Nikulin.
\newblock {Automorphic forms and Lorentzian Kac-Moody algebras. II}.
\newblock {\em International J. Math}, 9(1):201--275, 1998.

\bibitem{Kp}
A.W. Knapp.
\newblock {\em {Lie Groups Beyond an Introduction}}, volume 140 of {\em
  Progress in Mathematics}.
\newblock Birkhäuser, Boston, Basel, Berlin, 1996.

\bibitem{KK}
M.~Koecher and A.~Krieg.
\newblock {\em {Elliptische Funktionen und Modulformen}}, volume~1.
\newblock Springer-Verlag New York, Berlin, Heidelberg, 1997.
\newblock 2nd edition.

\bibitem{K2}
A.~Krieg.
\newblock {The graded ring of quaternionic modular forms of degree 2}.
\newblock {\em Math Z.}, 251:929--942, 2005.

\bibitem{R}
I.~Reiner.
\newblock {Real Linear Characters of the Symplectic Modular Group}.
\newblock {\em Proceedings of the American Mathematical Society},
  6(6):987--990, 1955.

\end{thebibliography}
\bibliographystyle{plain}

%
%
%
%
%
%
%
%
%
%
%
%
%
%

\end{document}